\documentclass[a4paper, 11pt, reqno]{amsart}
\usepackage{amsmath}
\usepackage{amsfonts}
\usepackage{amsthm}
\usepackage{epsfig}
\usepackage[pdftex]{hyperref}
\usepackage{indentfirst}

\newtheorem{teo}{Theorem}

\newtheorem{lem}{Lemma}
\newtheorem{prop}{Proposition}
\newtheorem{cor}{Corollary}

\newtheorem*{conj*}{Conjecture}

\newtheorem*{teop}{Main Theorem}

\usepackage{color}

\usepackage{enumerate}
\usepackage{enumitem}
\usepackage{float}	% permite posicionar a figura entre parágrafos usando H
\usepackage{amsthm}

\theoremstyle{definition}
\newtheorem*{ftm*}{Fuchsian $3$-Manifolds}
\newtheorem{remark}{Remark}

\title[On free boundary minimal annuli in geodesic balls of $\mathbb{S}^3_+$ and $\mathbb{H}^3$]{On uniqueness of free boundary minimal annuli in geodesic balls of $\mathbb{S}^3_+$ and $\mathbb{H}^3$ \ \ }
\author{C\'esar Lima}
\address{Instituto de Matem\'atica e Estat\'istica\\ Universidade Federal do Rio Grande do Sul\\ Porto Alegre - RS /  Instituto Federal do Cear\'a \\ Morada Nova - CE}
\email{cesar.lima@ifce.edu.br}

\begin{document}

\maketitle

\begin{abstract}
We consider $\Sigma$ an embedded free boundary minimal annulus in a geodesic ball in the round hemisphere $\mathbb{S}^3_+$ or in the hyperbolic space $\mathbb{H}^3$. Under the hypothesis of invariance due to an antipodal map on the geodesic ball and using the fact that this surface satisfies the Steklov problem with frequency, we prove that $\Sigma$ is congruent to a cri\-ti\-cal rotational annulus.
\end{abstract}

\providecommand{\abs}[1]{\lvert#1\rvert}

\linespread{1} % Espaçamento entre as linhas

\section{Introduction.}

Let $M$ be a Riemannian manifold with boundary and $\Sigma\subset M$ a compact submanifold that satisfies the boundary condition $\partial \Sigma\subset \partial M$. We say $\Sigma$ is a free boundary minimal submanifold when the mean curvature of the immersion is zero and that $\Sigma$ meets the boundary of $M$ orthogonally. In \cite{Ku-McG}, Kusner and McGrath  proved that a free boundary minimal annulus in the unit ball of $\mathbb{R}^3$, which is invariant by the antipodal map, is congruent to the critical catenoid. This is a particular case of the Fraser-Li conjecture \cite{Fr-Li}, which states that \textit{Up to congruence, the critical catenoid is the only embedded free boundary minimal annulus in $\mathbb{B}^3$.}

Our goal in this paper is to obtain a generalization of this result to spaces of constant curvature, more precisely, we will establish a characterization for an embedded free boundary minimal annulus in a geodesic ball in the round hemisphere and in the hyperbolic space, which is invariant by a map equivalent to the antipodal map in the geodesic ball. 

We recall the work of Otsuki \cite{Ot}, Mori \cite{Mo} and Carmo and Dajczer \cite{Ca-Da}, where they constructed a family of parametrized rotationaly invariant minimal surfaces in spaces of constant curvature. As observed by Li and Xiong \cite{Li-Xi} it is possible to restrict this family of parameterized surfaces to see the remaining elements, that we will call \textit{critical rotational annulus}, as free boundary minimal annulus in a geodesic ball. Denoting the coordinates of $\mathbb{R}^4$ as $(x,y)$, $x \in \mathbb{R}$ and $y \in \mathbb{R}^3$, due to the symmetries of spaces, we can focus on geodesic balls centered on the \textit{north pole}, the point $(1,0)$, which will be denoted by $\mathbb{B}^3_\varepsilon(r)$, where $r$ is the radius and the value of $\varepsilon$ indicates the round hemisphere, when $\varepsilon = 1$, and the hyperbolic space, when $\varepsilon = -1$. With this notation, the antipodal map in $\mathbb{B}^3_\varepsilon(r)$ is defined by $A(x,y) = (x,-y)$. Our goal is then to prove the following theorem:

\begin{teop}\label{teoprincipal}
Let $\Sigma\hookrightarrow \mathbb{B}^3_\varepsilon(r)$ be an embedded free boundary minimal annulus, where $r \in (0,\pi/2)$ if $\varepsilon = 1$. If $\Sigma$ it is invariant by the antipodal map in $\mathbb{B}^3_\varepsilon(r)$, then $\Sigma$ is congruent to a critical rotational annulus.
\end{teop}

A key property used in the proof of the Main Theorem is the connection between free boundary minimal surfaces and variations of the Steklov problem. For proper understanding, the Steklov problem is an eigenvalue problem with the spectral parameter in the boundary conditions, which consists in finding solutions of 
\begin{equation}\label{steklovclassical}
\left\{\begin{array}{ll}
\Delta u = 0, & \textrm{on} \ \Sigma,\\[5pt]
\dfrac{\partial u}{\partial \nu} = \sigma u, & \textrm{in} \ \partial\Sigma,
\end{array}\right.
\end{equation}
where $\Delta$ is the Laplace-Beltrami operator and $\nu$ is the outward unit vector field along $\partial\Sigma$. The theory of the Steklov problem is well established and was used by Fraser and Schoen \cite{Fr-Sc} to show that a free boundary minimal annulus in the closed unit ball in $\mathbb{R}^3$ is congruent to the critical catenoid, under the condition that the first Steklov eigenvalue is equal to one. In \cite{Ku-McG}, Kusner and McGrath observed that under the assumption of antipodal invariance the condition of the first Steklov eigenvalue follows from the two-piece property, which had recently been proven by Lima and Menezes \cite{Li-Me2} for the closed unit ball in $\mathbb{R}^3$. 

In \cite{Li-Me}, Lima and Menezes studied free boundary minimal surfaces in geodesic balls of the round hemisphere. They connected these objects to a Steklov problem with frequency, analyzing  maximizers for the quantity
$$\Theta_r(\Sigma,g) = [\sigma_0\cos^2r + \sigma_1\sin^2r]|\partial \Sigma|_g + 2|\Sigma|_g,$$
where $g$ is a metric in $\Sigma$, $|\Sigma|_g$ denotes the area of $\Sigma$ and $|\partial\Sigma|_g$ is the length of $\partial\Sigma$. In the case of an immersed $2$-dimensional manifold $\Sigma$ in a geodesic ball $\mathbb{B}^n_1(r)$ in the $n$-dimensional round sphere $\mathbb{S}^n$ the frequency is equal to $2$ and the Steklov problem is described as
\begin{equation}
\left\{\begin{array}{rl}
\Delta \varphi_i + 2\varphi_i = 0, & \textrm{in} \ \Sigma, \ i=0,1,\ldots,n,\\[6pt]
\dfrac{\partial\varphi_0}{\partial\nu} =- \tan r \,\varphi_0, & \textrm{on} \ \partial\Sigma,\\[9pt]
\dfrac{\partial\varphi_i}{\partial\nu} = \cot r \,\varphi_i, & \textrm{on} \ \partial\Sigma, \ i=1,\ldots,n,
\end{array}\right.
\end{equation}
where $\varphi_0,\ldots,\varphi_n$ are the coordinate functions. They showed in \cite[Theorem C]{Li-Me} that if $\Sigma$ is an annulus such that $\varphi_i$ is a $\sigma_1$-eigenfunction, for $i=1,\ldots,n$, then $n = 3$ and $\Sigma$ is a critical rotational annulus. This is analogous to the uniqueness result of the critical catenoid proved in \cite{Fr-Sc}, as well as to results of Montiel-Ros \cite{Mo-Ro} and El Soufi-Ilias \cite{El-Il} which caracterize the Clifford torus and the flat equilateral torus. Following the work of Lima and Menezes, Medvedev \cite{Me} used similar techniques to generalize some of their results to the case of surfaces in a geodesic ball in the hyperbolic space. In particular, in \cite[Theorem 5.10]{Me} he establishes that if $\Sigma$ is an annulus such that the coordinate functions $\varphi_i$, $i=1,\ldots,n$, are $\sigma_1$-eigenfunctions, where $\sigma_1$ is related to the Steklov problem in the hyperbolic space, then the surface is congruent to a critical rotational annulus. 

Fern\'andez, Hauswirth and Mira \cite{Fe-Ha-Mi} have recently constructed examples of immersed free boundary minimal annuli in the unit ball of $\mathbb{R}^3$ that are not congruent to the critical catenoid. In \cite{Ce-Fe-Mi}, Cerezo, Fern\'andez and Mira construct a family of free boundary annuli in geodesic balls of $\mathbb{S}^3$ and $\mathbb{H}^3$. Also, in \cite{Ce}, Cerezo construct a family of non-rotational free boundry minimal annuli in geodesic balls of $\mathbb{H}^3$. In view of these results, it was proposed in \cite{Ce-Fe-Mi,Me} the following conjecture:

\begin{conj*}
The critical rotational annuli are the only embedded free boundary minimal annuli in geodesic balls of $\mathbb{S}^3_+$ and $\mathbb{H}^3$.
\end{conj*}

$ $\\
\textbf{Acknowledgements:} I would like to thank professor Vanderson Lima for his dedication and assistance in this work. I would also like to thank the Federal University of Rio Grande do Sul and the Federal Institute of Cear\'a.

\section{The Steklov problem with frequency.}\label{Steklovproblem}

Let $(\Sigma,g)$ be a compact Riemannian manifold with boundary, denote by $\nu$ the outward unit vector field along $\partial\Sigma$ and consider $\Delta = \mathrm{div}\circ\nabla$ the Laplace-Beltrami operator. Fix a value $\alpha$ that is not in the spectrum of the operator $-\Delta$ with Dirichlet boundary condition. The study of the Steklov problem with frequency $\alpha$ uses the \textit{Dirichlet-to-Neumann operator with frequency $\alpha$}, $\mathcal{D}_\alpha$, which we will now define. Consider $u \in H^1(\Sigma)$. We say that $\Delta u \in L^2(\Sigma)$ if there is a function  $f \in L^2(\Sigma)$ such that
$$\int_\Sigma g(\nabla u,\nabla v) = -\int_\Sigma fv,$$
for any $v \in H^1_0(\Sigma)$. In this case, we denote $\Delta u = f$. When we consider $u \in H^1(\Sigma)$ we lose the sense of $u\vert_{\partial\Sigma}$ and the normal derivative of $u$ in $\partial\Sigma$. For the first case we can use the trace operator $\mathrm{T} : H^1(\Sigma)\rightarrow L^2(\partial \Sigma)$. For the second case, we can work with the normal derivative for functions $u \in H^1(\Sigma)$ with $\Delta u \in L^2(\Sigma)$. In this situation, we say that $\frac{\partial u}{\partial \nu} \in L^2(\partial\Sigma)$ if there is $\phi \in L^2(\partial\Sigma)$ such that 
$$\int_\Sigma [g(\nabla u,\nabla v) + v\Delta u] = \int_{\partial\Sigma} v\phi,$$
for all $v \in H^1(\Sigma)$. We denote $\phi = \frac{\partial u}{\partial \nu}$. Now, a $\hat{u} \in L^2(\partial\Sigma)$ is in $\mathrm{Dom}(\mathcal{D}_\alpha)$, the domain of the Dirichlet-to-Neumann operator, if there is a function $u \in H^1(\Sigma)$ that satisfies: $\Delta u = - \alpha u \in L^2(\Sigma)$,  $\mathrm{T}(u) = \hat{u}$ and $\frac{\partial u}{\partial \nu} \in L^2(\partial\Sigma)$. Finally, we define $\mathcal{D}_\alpha : \mathrm{dom}(\mathcal{D}_\alpha) \rightarrow L^2(\partial\Sigma)$ as
$${\mathcal{D}}_\alpha u := \frac{\partial u}{\partial \nu}.$$

To extract properties from ${\mathcal{D}}_\alpha$ we can consider $V := T(H^1(\Sigma)) \subset L^2(\partial\Sigma)$ and define $B : V\times V\rightarrow \mathbb{R}$ as
$$B(u,v) = \int_\Sigma [g(\nabla u,\nabla v)-\alpha uv].$$

In \cite{Ar-Ma2} it has been proven that $V = \mathrm{Dom}(\mathcal{D}_\alpha)$, that this space is dense in $L^2(\partial\Sigma)$ and that
$$\langle \mathcal{D}_\alpha u,v\rangle_{L^2(\partial\Sigma)} = B(u,v) \ \ \ \ \ \textrm{and} \ \ \ \ \ B(u,u) + a||u||^2_{L^2(\partial\Sigma)} \geq b ||u||^2_{H^1(\Sigma)}$$
for certain constants $a \geq 0$ e $b > 0$. These statements guarantee that $\mathcal{D}_\alpha$ is a self-adjoint operator which is bounded below and has compact resolvent. 

\subsection{Eigenspaces and eigenfunctions }\label{autoespacoseautofuncoes} 

It follows that the eigenvalues of $\mathcal{D}_\alpha$ form a non-decreasing sequence $(\sigma_i)_{i=0}^\infty\subset\mathbb{R}$ whose limit tends to infinity. The eingenvalue $\sigma_i$ is usually called the Steklov eingenvalue. A function that satisfies the Steklov problem (with frequency) is called a Steklov eigenfunction, or a Steklov $\sigma_i$-eigenfunction to highlight that $u$ satisfies $\frac{\partial u}{\partial \nu} = \sigma_i u$. We denote by $\mathcal{E}_i$ the eigenspace associated with the Steklov eigenvalue $\sigma_i$. For brevity, we often omit the term `Steklov'. There are several properties related to the Steklov problem that we want to use. To do so, we will assume the condition that
\begin{equation}\label{hipminimumvalue}
\alpha < \lambda_1^D
\end{equation}
where $\lambda_1^D$ is the first non-zero eigenvalue of the operator $-\Delta$ with Dirichlet boundary condition. We will call $\alpha$ a \textit{supercritical frequency} if $\alpha < \lambda_1^D$ and, as we will see, this condition is true for a free boundary minimal surface in a geodesic ball in both the round hemisphere and the hyperbolic space. We assume the terminology of supercritical frequency due to the change in complexity of the Steklov problem, where we also consider $\alpha = \lambda_1^D$ a critical frequency and $\alpha > \lambda_1^D$ a subcritical frequency.

\begin{lem}\label{properties} Under the hypothesis that $\alpha$ is a supercritical frequency the following properties hold:
\begin{enumerate}[label=\arabic*.,itemsep=0.25cm]
\item $\sigma_0$ is characterized by $\sigma_0 = \inf\{B(u,u) \ ; \ ||u||_{L^2(\partial\Sigma)} = 1\}$.
\item If $u \in \mathcal{E}_0$ then $u > 0$ or $u < 0$ on $\Sigma$. 
\item If $u \in \mathcal{E}_i$ and $u(p) = 0$ then for any neighborhood $V$ of $p$ there are points $q_1,q_2 \in V$ such that $u(q_1) > 0$ and $u(q_2) < 0$.
\item If $u$ is a $\sigma_i$-eigenfunction and $v$ is a $\sigma_j$-eigenfunction then either $\sigma_i = \sigma_j$ or $\langle u,v\rangle_{L^2(\partial\Sigma)} = 0$.
\end{enumerate}
\end{lem}

The statements \textit{1} and \textit{2} imply $\mathcal{E}_0$ is simple. Two important consequences of \textit{4} are that the eigenspaces $\mathcal{E}_i$ are orthogonal and that if $u$ is a positive eigenfunction then $u \in \mathcal{E}_0$.

\section{Nodal set of a Steklov eigenfunction.}\label{nodalsteklov2}

Exploring the duality between the Robin problem and the Steklov problem, Hassannezhad and Sher \cite{Ha-Sh} were able to prove a nodal count theorem for ${\mathcal{D}}_\alpha$. To understand this result, consider $u$ an eigenfunction. The \textit{nodal set} of $u$, denoted by $\mathcal{N}_u$, is defined as the set of zeros of $u$, $\mathcal{N}_u = \{p \in \Sigma \ ; \ u(p) = 0\}$, and a \textit{nodal domain} of $u$ is a connected component of $\Sigma\setminus\mathcal{N}_u$. 

%In \cite{Ha-Sh}, Hassannezhad and Sher explored the Steklov-Robin duality to establish the following result:

\begin{teo}[Hassannezhad-Sher]\label{hasssher}
Let $\Sigma$ be a compact Riemannian manifold. Fix $\alpha$ and consider $d$ the number of non-negative eigenvalues $\theta$  of the problem
\begin{equation}\label{countingmultiplicity}
\left\{\begin{array}{ll}
\Delta u +\alpha u = \theta u, & \textrm{on} \ \Sigma,\\[3pt]
u = 0, & \textrm{in} \ \partial\Sigma.
\end{array}\right.
\end{equation}

If $u$ is a $\sigma_i$-eigenfunction and $N_i$ represents the number of nodal domains of $u$ then $N_i \leq i + 1 + d$.
%\begin{equation}
%N_i \leq i + 1 + d.
%\end{equation}
\end{teo}

An important point is that we can characteriza the eigenvalues $\theta$ in terms of the eigenvalues of the Laplacian. The number of non-negative eigenvalues $\theta$ coincides with the number of Laplacian eigenvalues $\lambda$ that satisfy $\lambda-\alpha \leq 0$. In particular, if $\alpha < \lambda_1^D$ then every eigenvalue of \eqref{countingmultiplicity} is positive, which implies $d = 0$ in Theorem \ref{hasssher}. This observation will be used to prove the Proposition \ref{minimumvalue} in Section \ref{fbmmsf}.

We then have the following facts:

\begin{lem}\label{doisdominioseme1}
Every eigenfunction with $u \in \mathcal{E}_1$ has exactly two nodal domains.
\end{lem}
\begin{proof}[\bf Proof]
Since $\alpha$ is a supercritical frequency, we have that the number of nodal domains of $u \in \mathcal{E}_1$ is less than or equal to 2. Assume that $u$ only has one nodal domain. Item \textit{3} in Lemma \ref{properties} shows that $u > 0$ or $u < 0$ on $\Sigma$. Then, Item \textit{2} shows that $u \in \mathcal{E}_0$, a contradiction.
\end{proof}
 
\begin{lem}\label{bordoedominiosnodais}
Let $u$ be an eigenfunction with exactly two nodal domains $\Omega$ and $\Omega'$. Then
	\begin{enumerate}[label=\arabic*.]
	\item $\partial\Sigma \cap \Omega\not=\emptyset$ and $\partial\Sigma\cap\Omega' \not=\emptyset$.
	\item $u$ has opposite signs in $\Omega$ and $\Omega'$.
	\end{enumerate}
\end{lem}
\begin{proof}[\bf Proof]
First we will show that if $u$ is an eigenfunction with $u\vert_{\partial\Sigma}\equiv 0$ then $u\equiv 0$ on $\Sigma$. To see that we can use the fact the only solution of the problem
$$\left\{\begin{array}{ll}
\Delta v + \alpha v = 0, & \textrm{em} \ \Sigma,\\[3pt]
v = 0, & \textrm{em} \ \partial\Sigma,
\end{array}\right.$$
is the null function, since $\alpha$ is a supercritical frequency. Assume that $u$ is an eigenfunction with exactly two nodal domains. As we saw, $u$ cannot be zero on the boundary, from which we can assume that $\partial\Sigma\cap\Omega\not=\emptyset$. Furthermore, it follows from the Item \textit{2} in Lemma \ref{properties} that $u \not\in \mathcal{E}_0$. Consider $\phi \in \mathcal{E}_0$ a positive function in $\Sigma$. Since $u$ is an eigenfunction and $u\not\in \mathcal{E}_0$ we get that
$$0 = \langle u,\phi\rangle_{L^2(\partial\Sigma)} = \int_{\partial\Sigma} u\phi = \int_{\partial\Sigma\cap \Omega}u\phi + \int_{\partial\Sigma\cap\Omega'}u\phi.$$

Now, we have that $\partial\Sigma\cap\Omega$ is non-empty, $u$ does not change sign in this set and $\phi$ is positive, which guarantees that the first integral on the right side has a sign, which leads to what is desired.
\end{proof}

\subsection{Multiplicity bounds in a surface} 

In \cite{Ch}, Cheng studied the beha\-vior of eigenfunctions on a Riemannian manifold without boundary with the aid of spherical harmonic functions. In his work, was evident the complexity of dealing with the nodal set in any dimension, but in the case of surfaces he completely characterized the behavior of the nodal set. The techniques he presented were used in \cite{Ja,Ka-Ko-Po} to establish a multiplicity bounds of the Steklov eigenvalues in a compact surface with boundary. In \cite[Theorem 2.3]{Fr-Sc}, Fraser and Schoen presented a version of this result, whose strategy was used in \cite{Li-Me}, using Theorem \ref{hasssher}, to establish a version for the Steklov problem with frequency:

\begin{teo}[Lima-Menezes]\label{multiplicidadesigmai}
Let $\Sigma$ be a compact oriented surface with $\partial\Sigma\not=\emptyset$ and denote by $\gamma$ the genus of $\Sigma$. Assume that the Dirichlet eigenvalue problem \eqref{countingmultiplicity} does not admit non-negative eigenvalue. Then the multiplicity of a Steklov eingevalue (with frequency) $\sigma_i$ is at most $4\gamma + 2i + 1$.
\end{teo}

Our interest lies not only in the Theorem \ref{multiplicidadesigmai}, but also in some statements used in the strategy presented in \cite{Fr-Sc} regarding the nodal set. In the next propositions we will complement these statements to better understand the nodal set in an annulus. For what follows, we will assume the same hypotheses as the Theorem \ref{multiplicidadesigmai}. 

\begin{prop}\label{multiplicidadesigma1prop1}
Let $u$ be an Steklov eigenfunction. The set $S_u = \{p \in \Sigma \ ; \ u(p) = 0 \ \textrm{e} \ \nabla u(p) = 0\}$ is finite. Furthermore, if $p \in S_u\cap\partial\Sigma$ then $\mathcal{N}_u$, in a neighborhood of $p$, consists of a number $k$ , $2 \leq k < \infty$, of arcs that meet the boundary transversely.
\end{prop}
\begin{proof}[\bf Proof]
Following \cite{Li-Me} and \cite{Fr-Sc} the only point we need to note is that $k \geq 2$. The hypothesis $\nabla u(p) = 0$ guarantees that there is a harmonic polynomial $P(x,y)$ in the Taylor expansion of $u$ at $p$ that is not zero in the boundary and has degree greater than one. From the proof of Theorem 2.5 in \cite{Ch} we have that the nodal set of $P$ is formed by straight lines that pass through the origin and that the number of straight lines is equal to the degree of the harmonic polynomial, so we have at least two lines.
\end{proof}

\begin{prop}\label{multiplicidadesigma1prop2}
If $C$ is a connect component of $\partial\Sigma$ then the set $C\cap \mathcal{N}_u$ is finite, from each point of $C\cap\mathcal{N}_u$ there is at least one arc of $\mathcal{N}_u$, which is transversal to $C$, and the total of these arcs is an even number.
\end{prop}
\begin{proof}[\bf Proof]
Assume that $C\cap\mathcal{N}_u$ is infinite. The compactness of $C$ guarantees that there is a sequence $(q_n)$ in $C\cap\mathcal{N}_u$ with $q_n\rightarrow p \in C$. This implies $u(p) = 0$ and $du_p(V) = 0$, where $V \in T_p\partial\Sigma$. Since $\frac{\partial u}{\partial \nu}(p) = 0$ we get that $p \in S_u$. But from the Proposition \ref{multiplicidadesigma1prop1} there is a neighborhood $I$ of $p$ in $C$ such that $I\cap\mathcal{N}_u = \{p\}$, a contradiction. For the second statement, we have from the Lemma \ref{properties} that $u$ does not have isolated zeros, so $p \in C\cap\mathcal{N}_u$ it cannot be an isolated point. This shows that $p$ is at the closure of the nodal set of $u$ in $\mathrm{int}\,\Sigma$ which, according to \cite{Ch}, is a set of lines whose intersections occur at points of $S_u$. Still following \cite{Ch} we can show that this set of lines is finite, and with the fact that $p$ is in the closure of $\mathcal{N}_u\cap \mathrm{int}\Sigma$ guarantees what is desired. For the last statement, as the number of nodal lines is finite, we have that the number of arcs of $\mathcal{N}_u$ that start from $C$ is finite. Take $p \in C\cap\mathcal{N}_u$ and consider a number $k$ of arcs starting from $p$. Close to $p$ these arcs determine $k+1$ regions where $u$ changes the sign. If $k$ is even, we have an odd number of regions, so $u$ maintains the signal in a neighborhood of $p$ in $C$. If $k$ is odd, then $u$ changes the sign, so the total number of points from which an odd number of arcs starts is even.
\end{proof}

\begin{prop}\label{multiplicidadesigma1prop4}
If $u$ has, at most, two nodal domains then the order of vanishing of $u$ at $p \in \mathrm{int}\Sigma$ is less than or igual to one.
\end{prop}
\begin{proof}[\bf Proof]
The proof follows from Theorem 2.3 in \cite{Fr-Sc} with the only difference being the replacement of the nodal domain theorem \cite{Ku-Si} by the hypothesis that $u$ has at most two nodal domains.
\end{proof}

\subsection{The nodal set as a graph}

It follows from \cite{Ch} that if $\Sigma$ is a surface without boundary then $\mathcal{N}_u$ is a set of lines, each line being an immersed and closed submanifold with the property that when the lines meet they form an equiangular system. We can see a point where lines meet as a \textit{vertex} of a graph and the lines between vertices as \textit{edges}. But we can have two special types of edges: those that start from a vertex and are divergent (in the sense that the curve is not contained in any compact) on one side and those that are divergent from both sides. In our case, we have a compact surface with boundary, where we can apply this result to the interior of $\Sigma$. The continuity of $u$ shows that divergent lines tend to a zero in the boundary, which we will consider as a vertex of the graph. The Proposition \ref{multiplicidadesigma1prop2} guarantees that the number of vertices on the boundary is finite, that each vertex on the boundary has at least one edge and that the number of vertices in the interior of $\Sigma$ and the number of edges are finite. Thus, we can see $\mathcal{N}_u$ as a finite graph. 

Let $\mathcal{N}$ be a graph with the same properties as a nodal set. For a better description, we will use the following nomenclature: a \textit{maximal line} in $\mathcal{N}$ is the image of a topological embedding $f : [0,1]\rightarrow \mathcal{N}$ with the property that $f(0),f(1) \in \partial \Sigma$ and $f(t) \in \mathrm{int}\Sigma$, $t \in (0,1)$ (ex. the juxtaposition $A*D*C$ in the Figure \ref{figuraterminologia}); a \textit{circle} in $\mathcal{N}$ is the image of a topological embedding $f : \mathbb{S}^1\rightarrow \mathcal{N}$ (ex. the justaposition $A*E*F$ in the Figure \ref{figuraterminologia}); a \textit{line} in $\partial \Sigma$ is a topological embedding $f : I\rightarrow \partial\Sigma$, where $I = [0,1]$ or $I = \{0\}$; a \textit{closed chain} is the imagem of a circle in $\mathcal{N}\cap\mathrm{int}\Sigma$ or is the image of a topological embedding $f : \mathbb{S}^1\rightarrow \Sigma$ that is given by a (finite) juxtaposition of maximal lines in $\mathcal{N}$ and lines in $\partial \Sigma$, given alternately, so that if $C$ is a connected component of $\partial\Sigma$ then $\mathrm{Im}\,f\cap C$ is connected (ex. the justaposition $A*E*F$ or $B*X*C*D$ in the Figure \ref{figuraterminologia}). Observe that the structure of $\mathcal{N}$ guarantees that every edge of the graph is contained in a maximal line in $\mathcal{N}$ or in a circle in $\mathcal{N}$. In particular, we can walk in maximal lines and, if necessary, use lines in $\partial\Sigma$, in order to verify that every graph $\mathcal{N}$ admits a closed chain. 

\begin{figure}[H]%se quiser fixar basta usar [H]
\centering \includegraphics[scale=0.150]{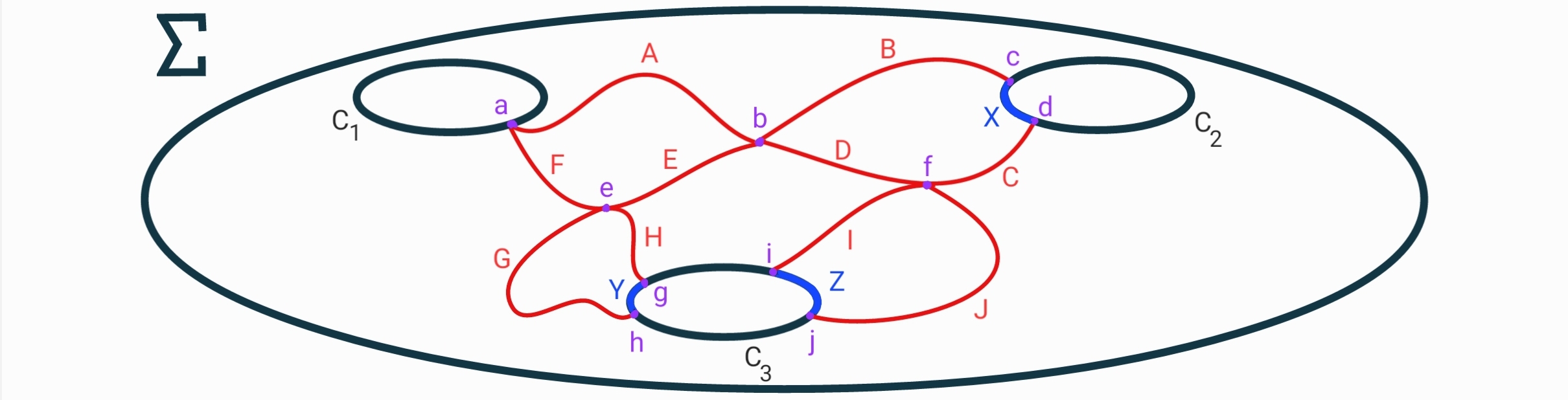}
\caption{$\Sigma$ is a plane region with the removal of three disks, $C_1$, $C_2$ and $C_3$. Lowercase and uppercase letters represent the vertices and edges of the graph, respectively. $X$, $Y$ and $Z$ are lines in the boundary.} 
\label{figuraterminologia}
\end{figure}

In order to obtain more information about $\mathcal{N}$ we can use the fact that $\Sigma$ is an orientable surface with genus zero. Take $f : \mathbb{S}^1\rightarrow \Sigma$ a closed chain and write $\Sigma\setminus\mathrm{Im}\,f = \Sigma_1\dot{\cup}\Sigma_2$. The last condition in the definition of a closed chain was adopted so that the $\mathrm{Im}\,f$ separated $\Sigma$ in two components. Topologically, $\overline{\Sigma}_1$ and $\overline{\Sigma}_2$ have the same properties as $\Sigma$, they are topological manifolds with boundary and genus zero. So, one direction we can take is to analyze whether $\mathcal{N}$ determines a graph in $\Sigma_i$, $i=1,2$. Since $\mathcal{N}$ has no isolated vertices, we can define $\mathcal{N}_i$ as the closure of $\mathcal{N}\cap\Sigma_i$, $i=1,2$. We will now prove some properties regarding the sets $\mathcal{N}_i$.

\begin{lem}\label{decompondon_i}
If $f : \mathbb{S}^1\rightarrow \Sigma$ is a closed chain then $\mathcal{N}_i = (\mathcal{N}\cap\Sigma_i)\,\dot{\cup}\,(\mathcal{N}_i\cap\mathrm{Im}\,f)$.
\end{lem}
\begin{proof}[\bf Proof]
The proof follows from the decomposition $\Sigma = \Sigma_1\dot{\cup}\mathrm{Im}\,f\dot{\cup}\Sigma_2$ and the fact that $\mathcal{N}$ is a closed set.
\end{proof}

\begin{lem}\label{arestanobordo}
Let $f : \mathbb{S}^1\rightarrow \Sigma$ be a closed chain. If $p \in \mathcal{N}_i\cap \mathrm{Im}\,f$ then there is an edge $R$ of $\mathcal{N}$ such that $p \in R$ and $R\cap\Sigma_i\not=0$.
\end{lem}
\begin{proof}[\bf Proof]
From Lemma \ref{decompondon_i}, there is a sequence $(q_n)$ in $\mathcal{N}\cap\Sigma_i$ such that $q_n\rightarrow p$. Since $\mathcal{N}$ is a finite graph, there is an edge $R$ which contains infinite points of the sequence. Since $R$ is closed, we get that $p \in R$.
\end{proof}

\begin{lem}\label{interiordaaresta}
Let $f : \mathbb{S}^1\rightarrow \Sigma$ be a closed chain. If $R$ is an edge of $\mathcal{N}$ with $R\cap\Sigma_i\not=\emptyset$, then $\mathrm{int}\,R\subset \Sigma_i$. As a consequence, we have that $\mathcal{N}_i\cap\mathrm{Im}\,f$ is formed by vertices of $\mathcal{N}$, therefore is finite. 
\end{lem}
\begin{proof}[\bf Proof]
Consider $p_1$ and $p_2$ the vertices of $R$. Take a point $p\in \mathrm{int}R\cap \Sigma_i$ and walk towards $p_1$. If we reach a point $q$ that is not in $\Sigma_i$ then $q \in \mathrm{Im}f$, where we will have that $q$ is a vertex of $\mathcal{N}$, since $f$ is a closed chain, which guarantees that $(p,p_1)$, the open segment (in $R$) from $p$ to $p_1$, is contained in $\Sigma_i$. If we don't reach that point we will have the same conclusion, so $(p,p_1)\subset \Sigma_i$. Repeating the process for $p_2$, we get that $\mathrm{int}R\subset \Sigma_i$. Now let us prove the second claim. The Lemma \ref{arestanobordo} guarantees that if $p \in \mathcal{N}_i\cap\mathrm{Im}\,f$ there is an edge $R$ of $\mathcal{N }$ with $p \in R$ and $R\cap\Sigma_i\not=\emptyset$. The first part of the proof guarantees that $\mathrm{int}\,R \subset \Sigma_i$, from which we obtain that $p$ is a vertex of $R$.
\end{proof}

Now we can show that $\mathcal{N}_i$ is a graph with the same properties as $\mathcal{N}$.

\begin{prop}\label{grafon_i}
Let $f : \mathbb{S}^1\rightarrow \Sigma$ be a closed chain. Then $\mathcal{N}_i$ is a graph formed by all edges $R$ of $\mathcal{N}$ such that $\mathrm{int}R\subset\Sigma_i$. Furthermore, we have that $\mathcal{N}_i$ has the same properties as $\mathcal{N}$.
\end{prop}
\begin{proof}[\bf Proof]
If $p \in \mathcal{ N}_i$ then the Lemmas \ref{decompondon_i} and \ref{arestanobordo} guarantee that there is an edge $R$ of $\mathcal{N}$ such that $p \in R$ and $R\cap\Sigma_i\not=\emptyset$. The Lemma \ref{interiordaaresta} guarantee that $\mathrm{int}\,R\subset \Sigma_i$. Since $\mathcal{N}$ is closed, we get that $R\subset \mathcal{N}_i$, which implies the first statement. This characterization of $\mathcal{N}_i$ shows that this graph is finite. Note that the interior and boundary of $\overline{\Sigma}_i$, as a topological manifold, are $\mathrm{int}\,\overline{\Sigma}_i = \Sigma_i\cap\mathrm{int}\,\Sigma$ and $\partial\overline{\Sigma}_i = F\cup C_1\cup\ldots\cup C_r$, where $C_1,\ldots,C_r$ are the components of the boundary of $\Sigma$ that are contained in $\Sigma_i$ and $F$ is a connected curve formed by the remaining points: if $f$ is the juxtaposition $f_1*\ldots*f_k$ then $F$ is given as $\tilde{f}_1*\ldots*\tilde{f}_k$, where $\tilde{f}_i = f_i$ if $f_i$ is an edge of $\mathcal{N}$ and $\tilde{f}_j = f_j^c$ if $f_j$ is a line in $\partial\Sigma$ and $f_j^c$ denotes the component of the boundary `facing' towards $\Sigma_i$.
\begin{figure}[H]%se quiser fixar basta usar [H]
\centering \includegraphics[scale=0.170]{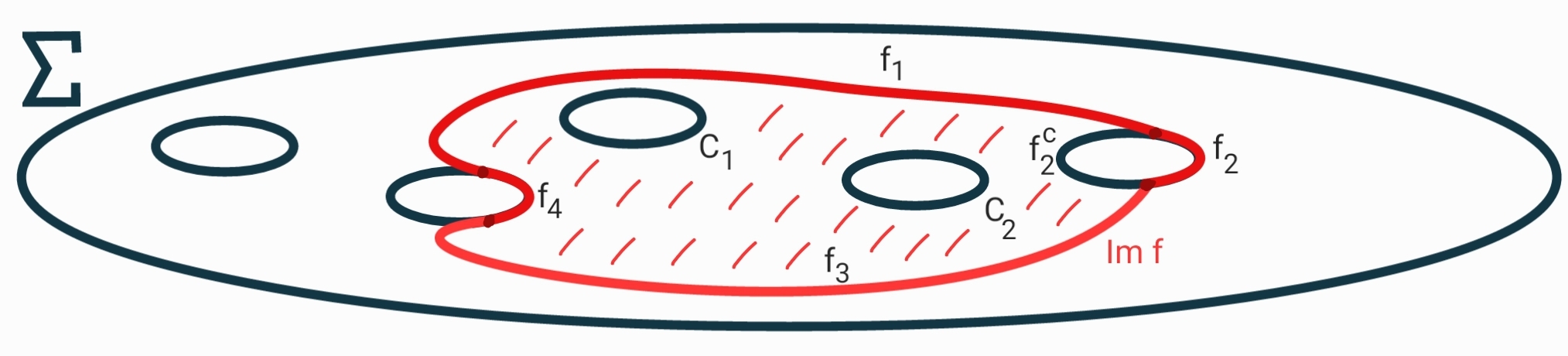}
\caption{The boundary of a component of $\Sigma\setminus \mathrm{Im}f$. Here, $\Sigma$ is a plane region with the removal of five open disks, $\Sigma_1$ is the highlighted region, $\mathrm{Im}f = f_1*f_2*f_3*f_4$ and $F = f_1*f_2^c*f_3*f_4$.} 
\label{descricaofronteirasigmai}
\end{figure}

If $p$ is a vertex of $\mathcal{N}_i$ in $\mathrm{int}\overline{\Sigma}_i$ then there is an even number of edges starting from $p$ and forming an equiangular system, since $\Sigma_i$ is opened in $\Sigma$. To finish the proof, we need to show that the total number of edges that starts in a component of the boundary is even. This is true for every element $C_i$, $i=1,\ldots,r$. To conclude the same for $F$, note that, since each vertex in the interior departs an even number of edges and each edge connecting two of these vertices is counted twice in this total, then the total of edges that start in a interior vertices and arives in the the boundary is an even number. Thus, it is immediate that the total number of edges departing from the boundary is an even number, which guarantees that $F$ has the same property.
\end{proof}

If $f$ is a closed chain the Proposition \ref{grafon_i} guarantees a decomposition of the graph $\mathcal{N}$, because $(\overline{\Sigma}_i,\mathcal{N}_i )$ has the same properties as $(\Sigma,\mathcal{N})$ and $A_{\mathcal{N}} = A_{\mathcal{N}_1}\,\dot{\cup}\,A_{\mathrm{Im}\,f}\,\dot{\cup}\,A_{\mathcal{N}_2}$, where $A_X$ is the set of edges of $\mathcal{N}$ included in the set $X$.

\begin{prop}\label{descrevendoN}
If $\Sigma\setminus\mathcal{N}$ has exactly two nodal domains and $f : \mathbb{S}^1\rightarrow \Sigma$ is a closed chain then $\mathcal{N}$ is equal to the union of the edges of $\mathcal{N}$ that are in the image of $f$.
\end{prop}
\begin{proof}[\bf Proof]
Writing $\Sigma\setminus\mathrm{Im}\,f = \Sigma_1\,\dot{\cup}\,\Sigma_2$ we get that $A_{\mathcal{N}} = A_{\mathcal{N}_1}\,\dot{\cup}\,A_{\mathrm{Im}\,f}\,\dot{\cup}\,A_{\mathcal{N}_2}$. Assume $A_{\mathcal{N}_i}\not=\emptyset$, $i\in \{1,2\}$. Since $(\overline{\Sigma}_i,\mathcal{N}_i)$ has the same properties as $(\Sigma,\mathcal{N})$ there is a closed chain $g : \mathbb{S}^1\rightarrow \overline{\Sigma}_i$. Now, $\overline{\Sigma}_i$ has genus zero, where can we write that $\overline{\Sigma}_i\setminus\mathrm{Im}\,g = A\,\dot{\cup}\,B$, which implies that $\Sigma\setminus \mathcal{N}$ has at least three domains, a contradiction.
\end{proof}

The Proposition \ref{descrevendoN} give us a format for the graph $\mathcal{N}$ when $\Sigma$ has genus zero, from which we extracted two patterns:
\begin{enumerate}[label=\roman*.]
\item $\mathcal{N}$ is the image of a circle $f : \mathbb{S}^1\rightarrow \Sigma$ with the property that the intersection $f(\mathbb{S}^1)\cap\partial\Sigma$ is empty;
\item there are distinct connected components $C_1,\ldots,C_r$ of the boundary of $\Sigma$ and maximal lines $f_1,\ldots,f_r$, so that $f_i$ connects $C_i$ to $C_{i+1}$ , where we call $C_{r+1} = C_1$, and with the property that two of these maximal lines can only intersect each other at their ends, where we get that $\mathcal{N} = \cup_{i=1}^r\mathrm{Im}f_i$.
\end{enumerate}

Topologically, we can see an annulus as a closed disk from which we have removed an open disk from its interior. Following the description above, we obtain the following topological possibilities for $\mathcal{N}$:
\begin{figure}[H] %se quiser fixar basta usar [H]
\centering \includegraphics[scale=0.12]{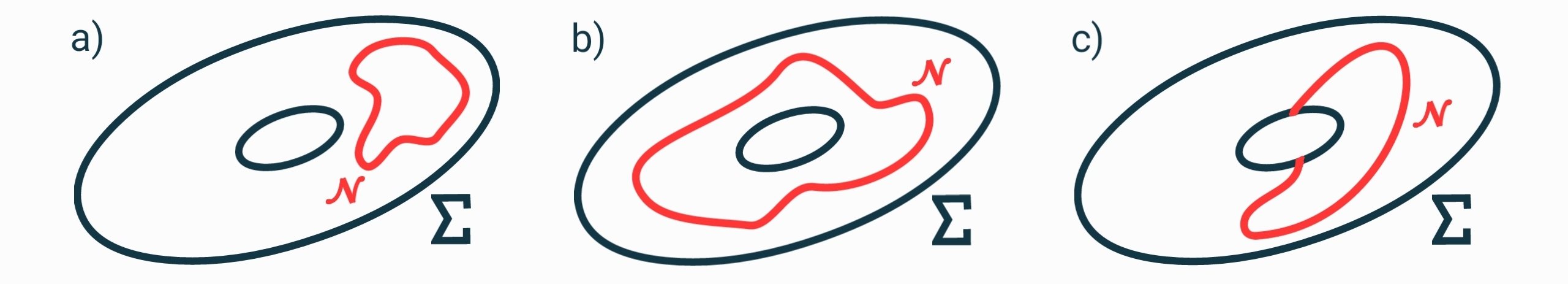}
\centering \includegraphics[scale=0.12]{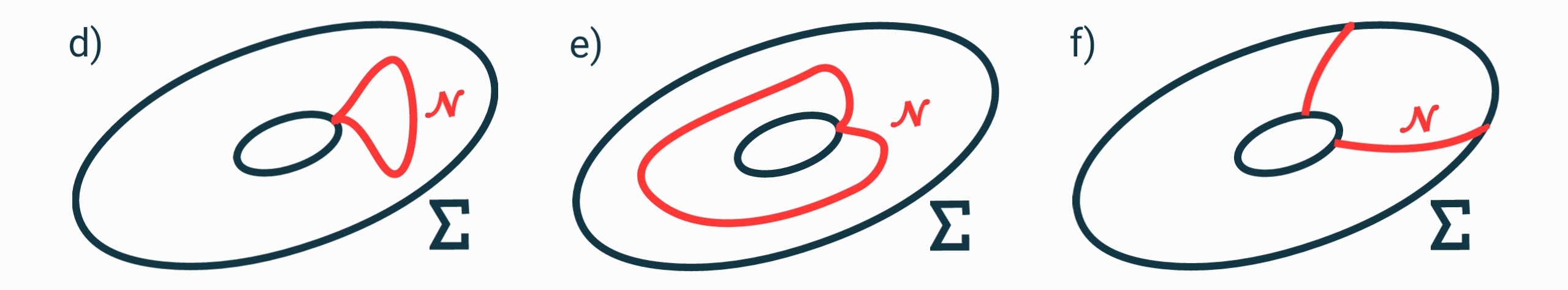}
\centering \includegraphics[scale=0.12]{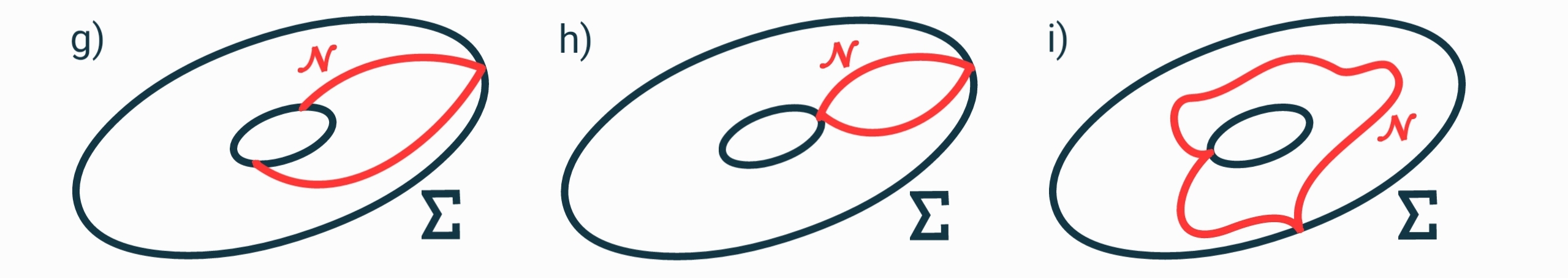}

%\centering \includegraphics[scale=0.1185]{./fig/modpart3.jpg}
%\caption{Possibilities for $\mathcal{N}$.}
\label{modpart1}
\end{figure}

%\begin{figure}[H] %se quiser fixar basta usar [H]
%\centering \includegraphics[scale=0.12]{figmodpart1.jpg}
%\centering \includegraphics[scale=0.12]{figmodpart2.jpg}
%\centering \includegraphics[scale=0.12]{figmodpart3.jpg}

%\centering \includegraphics[scale=0.1185]{./fig/modpart3.jpg}
%\caption{Possibilities for $\mathcal{N}$.}
%\label{modpart1}
%\end{figure}

%\begin{figure}[H] %se quiser fixar basta usar [H]
%\centering \includegraphics[scale=0.12]{figmodpart3.jpg}
%\caption{Possibilities for $\mathcal{N}$ (II).}
%\label{modpart2}
%\end{figure}

As a consequence of Lemma \ref{bordoedominiosnodais}, we can consider the following: $\partial\Sigma\cap \Sigma_1 \not=\emptyset$ and $\partial\Sigma\cap\Sigma_2\not=\emptyset$, which excludes the cases a), d) and h).

\section{Involutive isometries and the nodal set in an annulus.}

Let $\rho :\Sigma \rightarrow \Sigma$ and $u : \Sigma\rightarrow \mathbb{R}$ be an arbitraty maps. We say that $u$ is \textit{$\rho$-even} if $u\circ\rho = u$ and that $u$ in \textit{$\rho$-odd} if $u\circ \rho = -u$. We say that $u$ is \textit{$\rho$-invariant} if $u$ is $\rho$-even or $\rho$-odd. 
\begin{lem}\label{parenodal}
If $\rho : \Sigma\rightarrow \Sigma$ is a diffeomorphism and $u$ is an eigenfunction $\rho$-invariant then $\rho$ preserves the nodal set of $u$. Furthermore, if $u$ has exactly two nodal domains we have that if $u$ is $\rho$-even then $\rho$ preserves the nodal domains of $u$ and if $u$ is $\rho$-odd then $\rho$ interchanges the nodal sets of $u$.
\end{lem}
\begin{proof}[\bf Proof]
Since $u$ is $\rho$-invariant we have that $u$ is also $\rho^{-1}$-invariant, what gives us that $\rho(\mathcal{N}_u) = \mathcal{N}_u$. Since the nodal domains are connect and there are exactly two nodal domains it is immediate that $\rho$ preserve or interchange them. The result then follows from the fact that $u$ has opposite signs in the nodal domains (Lemma \ref{bordoedominiosnodais}).
\end{proof}

We say that a application $\rho : \Sigma\rightarrow \Sigma$ is \textit{involutive} if $\rho^2 = \mathrm{Id}_\Sigma$. We will now proceed in the direction of showing that an involutive isometry generates a decomposition of the eigenspace $\mathcal{E}_i$. Define $\rho_i : \mathcal{E}_i\rightarrow \mathcal{E}_i$ as $\rho_i(u) = u\circ\rho$. Choose a point $p \in \Sigma$ and denote $q = \rho(p)$. It is a standard calculation to verify that 
\begin{center}
$\Delta \rho_i(u)_p + \alpha\rho_i(u)_p = \Delta u_q + \alpha u(q) \ \ \ \ \ \textrm{and} \ \ \ \ \ d(\rho_i(u))_p(\nu) = du_q(\nu)$
\end{center}
which guarantees that $\rho_i$ is well defined. Next, it is clear that $\rho^{-1}$ determines an application $\rho^{-1}_i$ which is the inverse of $\rho_i$. Since both applications are linear we get that $\rho_i$ is an isomorphism. Since $\rho_i^2(u) = u\circ\rho^2 = u$ we conclude that $\rho_i$ is an involutive isomorphism in $\mathcal{E}_i$. Consider
$$\mathcal{A}_{i} := \ker (\mathrm{Id}_{\mathcal{E}_i} + \rho_i) \ \ \ \ \ \ \ \textrm{and} \ \ \ \ \ \ \ \mathcal{S}_{i} := \ker (\mathrm{Id}_{\mathcal{E}_i} -\rho_i)$$
which are the subspaces of the eigenfunctions $\rho$-odd and $\rho$-even, respectively. Since we can write a function $u$ as $\frac{1}{2}(u + \rho_i(u)) + \frac{1}{2}(u-\rho_i(u))$ and $\rho_i$ is involutive, it is immediate that
\begin{equation}
\mathcal{E}_{i} = \mathcal{A}_{i}\oplus \mathcal{S}_{i}
\end{equation}

Using an isometry with an antipodal behavior we will establish a cha\-racte\-rization of the eigenspace $\mathcal{E}_1$, which is one of the central pieces for the proof of the Main Theorem. For the proof we will follow some of the ideas presented by Kusner and McGrath \cite{Ku-McG}, in their work with the Steklov problem, along with the possibilities obtained for the nodal set in an annulus for the Steklov problem with frequency $\alpha$.

\begin{teo}\label{generale1c}
Assume $\Sigma$ is an annulus admitting an orientation-reversing isometry $\rho$, which is involutive and has no fixed points. Furthermore, suppose that there is a three-dimensional space $\widetilde{\mathcal{C}}$ formed by $\rho$-odd Steklov eigenfunctions with frequency $\alpha$, where $\alpha$ is a supercritical frequency, such that each nonzero function in $\tilde{C}$ has exactly two nodal domains. Then $\mathcal{E}_{1} = \widetilde{\mathcal{C}}$.
\end{teo}
\begin{proof}[\bf Proof]
Step 1: we will show that we can choose a topological circle $C$ in $\mathrm{int}\Sigma$ such that $\Sigma\setminus C = A\dot{\cup}B$ with $\rho(C) = C$ and $\rho(A)=B$. Take $p \in \Sigma$ and consider $\beta : [0,1] \rightarrow \Sigma$ a continuous curve that connects $p$ to $\rho(p)$. We can restrict the domain of $\beta$ to an interval $[a,b]\subset [0,1]$ so that $\rho(\beta(a)) = \beta(b)$ and $\rho(\beta(t))\not=\beta(s)$, for any $t,s \in (a,b)$. To see this, consider $X = \{t \in [0,1] \ ; \ \exists \ s > t \ ; \ \rho(\beta(t)) = \beta(s)\}$. Choosing an increasing sequence that converges to $a:=\sup X$ we will have, from the definition of $X$, another sequence that converges to a value $b \geq a$. Since $\rho$ is fixed-point free it is immediate that $b > 0$. From the definition of $X$ we will have that $[a,b]$ is the desired interval. 

Now, consider $p \in \mathrm{int}\Sigma$ and $\beta : [0,1]\rightarrow \mathrm{int}\,\Sigma$ a smooth, unit speed, injetive curve that connects $p$ to $\rho(p)$ and with the property that $\rho(\beta(t))\not=\beta(s)$, for $t,s\in(0,1)$. Define $\gamma := \rho\circ\beta$ and call $C$ the image of the juxtaposition $\beta*\gamma$. We have that $\rho(C) = C$ and, since $\Sigma$ is orientable with genus zero, $\Sigma\setminus C = A\dot{\cup}B$, where $A$ and $B$ are open connected sets. Consider $J_1$ a unit vector field along $\beta$ such that $\{\beta'(t),J_1(t)\}$ is an orthonormal set and such that $J_1(t)$ points to the set $B$, for $t \in (0,1)$. Define $J_2$ in the same way, but with respect to $\gamma$. Observe that $\{\beta',J_1\}$ and $\{\gamma',J_2\}$ have the same orientation. Choose $t_0\in (0,1)$. As $\rho$ reverses the orientation and $\rho_*(\beta'(t_0)) = \gamma'(t_0)$ we have that $\rho_*(J_1(t_0)) = -J_2(t_0)$, which implies that $\rho(A) = B$.

Step 2: since we can decompose $\mathcal{E}_{1} = \mathcal{A}_{1}\oplus \mathcal{S}_{1}$, we will use the first statement to show, by contradiction, that $\mathcal{S}_1 = \{0\}$. Assume there is a nonzero function $u \in \mathcal{S}_1$.  By hypothesis, $u$ has exactly two nodal domains $\Omega$ and $\Omega'$. Consider $\beta$ as in the first step with the condition that the imagem of $\beta$ is in $\Omega\cap \mathrm{int}\,\Sigma$. As $u$ is a $\rho$-even function we have from the Lemma \ref{parenodal} that $\rho(\Omega) = \Omega$ and $\rho(\Omega') = \Omega'$. It follows that $C\subset \Omega$ and, in particular, we get that $\Omega'$ must be contained in $A$ or in $B$. If $\Omega'\subset A$ then $\Omega'= \rho(\Omega')\subset \rho(A) = B$, a contradiction. The same goes for $\Omega'\subset B$. So, $\mathcal{E}_1 = \mathcal{A}_1$.

Step 3: if we show that $\tilde{\mathcal{C}}\subset \mathcal{A}_1$, we will complete the proof, since we have, by hypothesis, that $\dim\tilde{\mathcal{C}} = 3$, that $\mathcal{E}_1 = \mathcal{A}_1$ and, from Theorem \ref{multiplicidadesigmai}, that $\dim\mathcal{E}_1 \leq 3$, which guarantees $\mathcal{E}_1 = \tilde{\mathcal{C}}$. As $\tilde{\mathcal{C}}$ is a subspace formed by eigenfunctions, there exists $i$ such that $\tilde{\mathcal{C}}\subset \mathcal{E}_i$. To prove that $\tilde{\mathcal{C}}\subset \mathcal{E}_1$ we just need to show, according to the orthogonality of eigenspaces (Lemma \ref{properties}), that there are $u\in\mathcal{E}_1$ and $\phi\in\tilde {\mathcal{C}}$ such that $\langle u,\phi\rangle_{L^2(\partial\Sigma)}\not=0$. Before going in this direction, we will first, through the existence of the isometry, reduce the number of possibilities of the nodal set. By hypothesis, every function $\phi \in \tilde{\mathcal{C}}$ has exactly two nodal domains, whereas the Lemma \ref{doisdominioseme1} guarantees the same for a function $u \in \mathcal{E}_1$. Then, the possibilities obtained for the nodal set are valid for eigenfunctions in both sets. Call $B$ and $B'$ the components of the boundary of $\Sigma$ and consider $\mathcal{N}$ a nodal set. Since $\rho$ reverses the orientation, we have that $\rho(B) = B'$. In particular, $\rho(\mathcal{N}\cap B) = \mathcal{N}\cap B'$, so the number of vertices of $\mathcal{N}$ in $B$ is equal to the number of vertices of $\mathcal{N}$ in $B'$. This fact excludes the cases c), e) and g) described in Figure \ref{modpart1}, leaving only the following three possibilities:
\begin{figure}[H] %se quiser fixar basta usar [H]
\centering \includegraphics[scale=0.125]{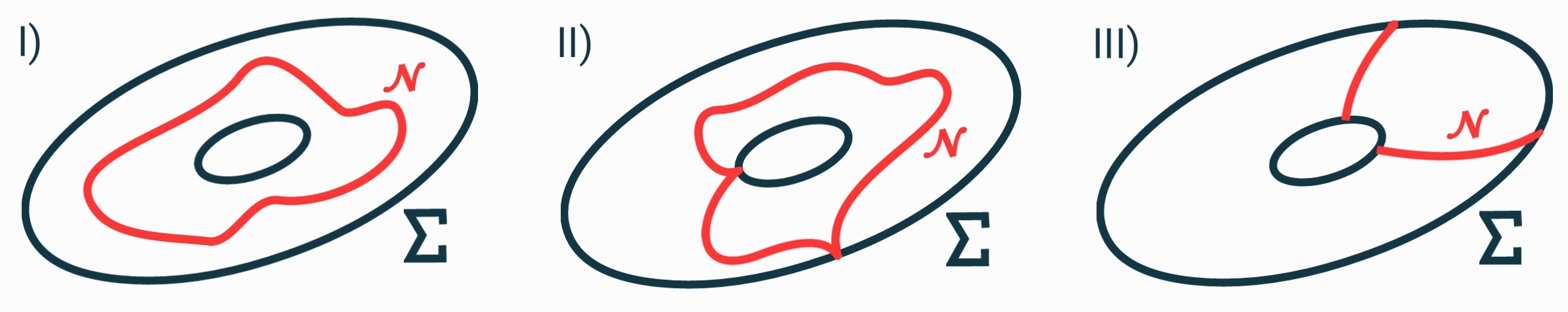}
\caption{Possibilities for the nodal set.}
\label{modelosfinais}
\end{figure}

Another information we will need is that functions in $\mathcal{E}_1$ and $\tilde{\mathcal{C}}$ invert the sign in $\partial\Sigma$. We will now show that for every $u \in \mathcal{E}_1$ there is $\phi \in \tilde{\mathcal{C}}$ with $\langle u,\phi\rangle_{L^2(\partial\Sigma)}\not=0$, which guarantees what we want. To do so, we can analyze whether or not $u$ changes sign in a connected component of $\partial\Sigma$. Assume first that $u$ does not change sign in the components of $\partial\Sigma$. It follows that $\mathcal{N}_u$ has one of the first two formats in the Figure \ref{modelosfinais}. Choose $p \in \partial\Sigma$ any point, take $V \in T_p\partial\Sigma$ unitary and define $F : \tilde{\mathcal{C}}\rightarrow \mathbb{R}^2$ by $F(\phi) = (\phi(p),d\phi_p(V))$. The linearity of $F$ guarantees that there is a non-zero $\phi$ with $\phi(p) = 0$ and $d\phi_p(V) = 0$. Since the normal derivative of $\phi$ in $p$ is zero, we conclude that $\nabla\phi(p) = 0$. The Proposition \ref{multiplicidadesigma1prop1} then guarantees that there are at least two arcs of $\mathcal{N}_\phi$ starting from $p$, which shows that this set has the format II in the Figure \ref{modelosfinais}. It follows that, with the exception of the points $p$ and $\rho(p)$, one of the functions $\phi$ and $-\phi$ has the same sign as $u$, which implies that $\langle u,\phi \rangle_{L^2(\partial\Sigma)}\not=0$. Now suppose that $u$ changes sign in one of the components of $\partial\Sigma$. It follows that $\mathcal{N}_u$ has the format III in the Figure \ref{modelosfinais}. In particular, we can consider $p,q$ points in a component of $\partial\Sigma$ where $u(p) = u(q) = 0$. Defining a linear transformation $F : \tilde{\mathcal{C}}\rightarrow \mathbb{R}^2$ by $F(\phi) = (\phi(p),\phi(q))$ we conclude that there is $\phi \in \tilde{\mathcal{C}}$ non-zero with $\phi(p) = \phi(q) = 0$, therefore $\mathcal{N}_\phi$ also follows possibility III. Hence, $\phi$ or $-\phi$ has the same sign as $u$ in $\partial\Sigma$, which shows that $\langle u,\phi\rangle_{L^2(\partial\Sigma)}\not=0$.
\end{proof}

\section{Free boundary minimal submanifolds in $\mathbb{B}^n_\varepsilon(r)$.}\label{fbmmsf}

Our goal here is to study the Steklov problem (with frequency) on a geodesic ball in a space form. Therefore, we will introduce a notation that embrace the cases of both the unit sphere and the hyperbolic space. A point $p \in \mathbb{R}^{n+1}$ will be described as $p = (x,y)$, where $x\in \mathbb{R}$ e $y = (y^1,\ldots,y^n) \in \mathbb{R}^n$. The canonical base of $\mathbb{R}^{n+1}$ will be represented as $\{\partial_0,\partial_1,\ldots,\partial_n\}$, where $\partial_0 = (1,0)$. Also, we will limit the Einstein summation convention to the last $n$ coordinates, so that $p = (x,y)$ can be represented as $p = x\partial_0 + y^i\partial_i$. Now, denote by $\delta$ and $\eta$ the euclidean and the Minkowski metrics in $\mathbb{R}^{n+1}$, respectively. Our focus is to deal with the following spaces: \textit{i)} the hemisphere in the round sphere of radius one, characterized by the condition $x> 0$, immersed in $(\mathbb{R}^{n+1},\delta)$; \textit{ii)} the hyperbolic space, in the model of the upper sheet of the hyperboloid, immersed in $(\mathbb{R}^{n+1},\eta)$.  So, let $\epsilon \in \{-1,1\}$. We consider a \textit{space form} the subset $\mathbb{E}^n_\varepsilon \subset \mathbb{R}^{n+1}$ with metric $\overline{g}_\varepsilon$ in $\mathbb{R}^{n+1}$ defined by
\begin{equation}
\mathbb{E}^n_\varepsilon = \{(x,y) \in \mathbb{R}^{n+1} \ ; \ x^2 + \varepsilon||y||^2 = 1 \ \textrm{and} \ x > 0\}
\end{equation}
where $||y||$ represent the euclidean norm of the vector $y \in \mathbb{R}^n$, and
\begin{equation}
\overline{g}_\varepsilon = \varepsilon dx^2 + \sum_{i=1}^n(dy^i)^2.
\end{equation}

For convenience, we will omit the subindex in the metric, leaving it indexed only in the set, and use the notation $\overline{g} = \langle\ ,\,\rangle$ when necessary. The next step is to work with a closed geodesic ball in $\mathbb{E}^n_\varepsilon$. Due to the symmetries of this model, we will choose a distinct point which will help in describing the Steklov problem. There is a clear choice, the point $\partial_0$, which we will sometimes call the \textit{north pole}. We will denoted by $\mathbb{B}^n_\varepsilon(r)$ the closed geodesic ball of radius $r$ and center $\partial_0$ in $\mathbb{E}^n_\varepsilon$. When $\varepsilon = 1$ we will assume $r \in (0,\pi/2)$. To continue towards a single notation, we will index the trigonometric functions with the value $\varepsilon$, like $\sin_\varepsilon$, $\cos_\varepsilon$, etc., so that if $\varepsilon = 1$ we get the usual trigonometric functions and if $\varepsilon = -1$ we get the hyperbolic trigonometric functions. It's immediate that
\begin{equation}
\mathbb{B}^n_\varepsilon(r) = \{(x,y) \in \mathbb{E}^n_\varepsilon \ ; \ ||y|| \leq \sin_\varepsilon r\}
\end{equation}
and
\begin{equation}
\partial\mathbb{B}^n_\varepsilon(r) = \{(x,y) \in \mathbb{E}^n_\varepsilon \ ; \ ||y|| = \sin_\varepsilon r\}.
\end{equation}

Established the basic notation of a space form, we will study how to characterize a free boundary minimal submanifold of $\mathbb{B}^n_\varepsilon(r)$. Consider $\Sigma$ a $k$-dimensional immersed submanifold of the closed geodesic ball and denote by $g$ the Riemannian metric of $\Sigma$. Assume $\partial\Sigma\not=\emptyset$. Let $v \in \mathbb{R}^{n+1}$. The coordinate functions are defined by $\varphi_v : \Sigma\rightarrow \mathbb{R}$ so that
\begin{equation}
\varphi_v(x) = \langle v,x\rangle,
\end{equation} 
and, for convinience, when $v = \partial_i$ is a coordinate vector we prefer to use the simplified notation $\varphi_i$ instead of $\varphi_{\partial_i}$. Next, we say that a minimal immersion $\Sigma\hookrightarrow \mathbb{B}^n_\varepsilon(r)$ is \textit{free boundary} if $\partial\Sigma = \Sigma\cap \partial\mathbb{B}^n_\varepsilon(r)$ and $\nu = N\vert_{\partial\Sigma}$ where $\nu$ is the unity outward vector field along $\partial\Sigma$ and $N$ is the unity outward vector field along $\partial \mathbb{B}^n_\varepsilon(r)$. We have the following characterization established in \cite{Li-Me,Me}:

\begin{prop}\label{coordinatefunctions}
Consider $\Sigma\hookrightarrow \mathbb{B}^n_\varepsilon(r)$ an immersion such that $\partial \Sigma = \Sigma\cap\partial\mathbb{B}^n_\varepsilon(r)$ and let $\nu$ be the unity outward vector field along $\partial \Sigma$. Then $\Sigma$ is an immersed free boundary minimal submanifold of $\mathbb{B}^n_\varepsilon(r)$ if and only if the coordinate functions satisfy
\begin{equation}
\left\{\begin{array}{rl}
\Delta \varphi_i + k\varepsilon \varphi_i = 0, & \textrm{in} \ \Sigma, \ i=0,1,\ldots,n,\\[6pt]
\dfrac{\partial\varphi_0}{\partial\nu} =-{\varepsilon} \tan_\varepsilon r \,\varphi_0, & \textrm{on} \ \partial\Sigma,\\[9pt]
\dfrac{\partial\varphi_i}{\partial\nu} = \cot_\varepsilon r \,\varphi_i, & \textrm{on} \ \partial\Sigma, \ i=1,\ldots,n.
\end{array}\right.
\end{equation}
\end{prop}

When $\Sigma$ is compact with non empty boundary, a hypothesis that we will assume from now on, the Proposition \ref{coordinatefunctions} establishes that the coordinate functions are eingenfunctions of a Steklov problem with frequency $k\varepsilon$. More precisely, we will deal with the problem
\begin{equation}\label{Steklovproblemspaceform}
\left\{\begin{array}{ll}
\Delta u + \varepsilon k u = 0, & \textrm{in} \ \Sigma, \\[6pt]
\dfrac{\partial u}{\partial \nu} - \sigma u = 0, & \textrm{on} \ \partial\Sigma
\end{array}\right.
\end{equation}
where $u \in C^\infty(\Sigma)$. In order to apply the theory of Sections \ref{Steklovproblem} and \ref{nodalsteklov2} we need additional information about the frequency. 

\begin{prop}\label{minimumvalue}
Consider $(\Sigma^k,g)$ a compact Riemannian manifold and assume $\Sigma\hookrightarrow \mathbb{B}^n_\varepsilon(r)$ is a free boundary minimal immersion. Then, the Dirichlet problem
\begin{equation}\label{problemadedirichlet}
\left\{\begin{array}{ll}
\Delta u + \varepsilon ku = \lambda u, & \textrm{in} \ \ \Sigma,\\[4pt]
u = 0, & \textrm{on} \ \ \partial\Sigma
\end{array}\right.
\end{equation}
has no solution for any $\lambda \geq 0$. 
\end{prop}
\begin{proof}[\bf Proof]
A non-trivial solution of \eqref{problemadedirichlet} means that the Laplace problem with Dirichlet condition, $-\Delta u = (\varepsilon k-\lambda)u$ e $u\vert_{\partial\Sigma} = 0$, admits a non-trivial solution. In this case, $\varepsilon k-\lambda$ is an eigenvalue. As such eigenvalues are positive, it's immediate that $\lambda < 0$ when $\varepsilon=-1$. The case $\epsilon = 1$ was proved in \cite{Li-Me}. 
\end{proof}

\begin{remark}
Since $\varphi_0$ in a positive eigenfunction we have that $\sigma_0 = -\tan_\varepsilon r$.
\end{remark}
 
In the remain of the section we will state some results necessary to demonstrate the Main Theorem. In what follows, we will consider
\begin{equation}
\mathcal{C} = \langle\{\partial_1,\ldots,\partial_n\}\rangle
\end{equation}
i.e. $\mathcal{C}$ is the vector space genereted by the vectors $\partial_i$, $i=1,\ldots,n$.

\subsection{The two-piece property}

A key tool in \cite{Ku-McG} is the two-piece property established by Lima and Menezes \cite{Li-Me2}. Recently, Naff and Zhu \cite{Na-Zh} proved the following result:

\begin{teo}[Naff-Zhu]
If $\Sigma_0,\Sigma_1 \subset \mathbb{B}^n_\varepsilon(r)$ are free boundary embedded minimal hypersurfaces in a closed geodesic ball of radius $r \in (0,\mathrm{diam}(\mathbb{E}^n_\varepsilon))$ then 
$$\mathbb{B}^n_{\varepsilon,+}(r)\cap \Sigma_0\cap\Sigma_1 \not=\emptyset,$$
where $\mathbb{B}^3_{\varepsilon,+}(r)\subset \mathbb{B}^3_\varepsilon(r)$ is a closed half-ball.
\end{teo}

For our purposes, we are interested in the following consequence:

\begin{cor}[Two-piece property for $\mathbb{B}^n_\varepsilon(r)$]\label{twopieceproperty}
Let $\Sigma$ be a compact embedded free boundary minimal hypersurface in $\mathbb{B}^n_\varepsilon(r)$. If $v \in \mathcal{C}$ then $$\Sigma\cap H_+(v) \ \ \ \ \ \ \ \textrm{and} \ \ \ \ \ \ \ \Sigma\cap H_-(v)$$ are connect sets, where $H_+(v) = \{w \in \mathbb{R}^{n+1} \ ; \ \overline{g}(v,w) > 0\}$ and $H_-(v) = \{w \in \mathbb{R}^{n+1} \ ; \ \overline{g}(v,w) < 0\}.$ 
\end{cor}

\subsection{Critical rotational annuli}\label{families}

We will now describe two families of free boundary minimal annuli in $\mathbb{B}^3_\varepsilon(r)$, one for each value of $\varepsilon$. The construction of these families goes back to the work of Otsuki \cite{Ot} and Mori \cite{Mo}, who constructed a family of rotational minimal surfaces in $\mathbb{S}^3$ and $\mathbb{H}^3$. In \cite{Ca-Da}, Carmo and Dajczer generalized the results of Otsuki and Mori, where they constructed a family of hypersurfaces in spaces of constant curvature, which we will now describe. First, consider $f_\varepsilon,g_\varepsilon : \mathbb{R} \rightarrow \mathbb{R}$ functions defined as $f_\varepsilon(s) = \sqrt{1 -\varepsilon a \cos_\varepsilon(2s)}$ and $g_\varepsilon(s) = \sqrt{\varepsilon + a\cos_\varepsilon(2s)}$ where $a \in (-1,0]$, if $\varepsilon = 1$, and $a \in (1,\infty)$, if $\varepsilon = -1$. Now, consider $\psi_\varepsilon : \mathbb{R}\rightarrow \mathbb{R}$ defined as
$$\psi_\varepsilon(s) = \sqrt{\varepsilon\left(1-a^2\right)}{\displaystyle\int_0^s}\dfrac{1}{2\,f_\varepsilon(s)^2g_\varepsilon(s)}dt.$$

Finally, we define $\Phi_{a,\varepsilon} : \mathbb{R}\times \mathbb{S}^1\rightarrow \mathbb{E}^3_\varepsilon$ given by
$$\Phi_{a,\varepsilon}(s,\theta) = \left(\dfrac{}{}f_\varepsilon(s)\cos_\varepsilon\psi_\varepsilon(s),f_\varepsilon(s)\sin_\varepsilon\psi_\varepsilon(s),g_\varepsilon(s)\cos\theta,g_\varepsilon(s)\sin\theta\right).$$

As observed by Li and Xiong \cite{Li-Xi}, there exists a value of $a$ and a positive number $s_0 \in \mathbb{R}$ such that
\begin{equation}
\Phi_{a,\varepsilon} : [-s_0,s_0]\times\mathbb{S}^1\rightarrow \mathbb{B}^3_\varepsilon(r)
\end{equation}
is a free boundary minimal annulus in $\mathbb{B}^3_\varepsilon(r)$. It is important to note that when $\varepsilon = 1$ we have a limitation on the radius of the geodesic ball: $r \in (0,\frac{\pi}{2})$. For what follows, we will call an element in this family a \textit{critical rotational annulus}. In \cite[Theorem C]{Li-Me} Lima and Menezes presented a classification result for free boundary minimal annulus in a geodesic ball $\mathbb{B}^n_\varepsilon(r)$ in $\mathbb{S}^3_+$, with the limitation of the radius of the geodesic ball $r \in (0,\frac{\pi}{2})$, using the hypothesis that the coordinate functions are first eigenfunctions of the Steklov problem. Following their work, Medveded \cite[Theorem 5.10]{Me} obtained, using the same hypothesis about $\sigma_1$, a similar result but for free boundary minimal annulus in a geodesic ball $\mathbb{B}^n_\varepsilon(r)$ in $\mathbb{H}^3$. These two results can be combined into the following theorem:

\begin{teo}[Lima-Menezes \cite{Li-Me}/Medvedev \cite{Me}]\label{LMM}
Let $\Sigma$ be an annulus and consider $\Phi = (\Phi_0,\ldots,\Phi_n) : (\Sigma,g)\rightarrow \mathbb{B}^n_{\varepsilon}$ a free boundary minimal immersion. If $\Phi_i$ is a $\sigma_1$-eigenfunction, for $i=1,\ldots,n$, then $n = 3$ and $\Phi(\Sigma)$ is a critical rotational annulus.
\end{teo}

\section{Proof of the Main Theorem.}

In what follows, $n = 3$. Let $\Sigma$ be an embedded surface in the geodesic ball $\mathbb{B}^3_\varepsilon(r)$, where we consider $r \in (0,\pi/2)$ if $\varepsilon = 1$. We define the \textit{antipodal map} in $\mathbb{B}^3_\varepsilon(r)$ as the involutive isometry $A : \mathbb{B}^3_\varepsilon(r) \rightarrow \mathbb{B}^3_\varepsilon(r)$, $A(x,y) = (x,-y)$. As before, consider $\mathcal{C}$ the vector space generated by the set $\{\partial_1,\partial_2,\partial_3\}$.

\begin{lem}\label{CisAodd}
If $\Sigma$ is a surface in $\mathbb{B}^3_\varepsilon(r)$ invariant by the antipodal map then $\varphi_v$ is $A$-odd for all $v \in \mathcal{C}$.
\end{lem}

Returning to the Theorem \ref{generale1c}, when we consider an embedded free boundary minimal annulus in $\mathbb{B}^3_\varepsilon(r)$ it is necessary that we have two conditions about the antipodal map to apply this result. First, the restriction of $A$ to the annulus cannot have fixed points and, second, the antipodal map needs to interchange the orientation of the annulus. We will now prove that these statements remain true for the Steklov problem with frequency. This was the strategy adopted by Fraser and Schoen \cite{Fr-Sc} and later, using the two-piece property in the unit ball \cite{Li-Me2}, by Kusner and McGrath \cite{Ku-McG}, which we will follow to prove the Main Theorem.

\begin{proof}[\bf Proof of the Main Theorem]
Since the annulus $\Sigma$ is invariant by the antipodal map in $\mathbb{B}^3_\varepsilon(r)$ we can consider the map $A_\Sigma : \Sigma\rightarrow \Sigma$, the restriction of the antipodal map to the annulus $\Sigma$. 

Step 1: the two-piece property (Corollary \ref{twopieceproperty}) guarantees that for all $v \in \mathcal{C}$ the eigenfunction $\varphi_v$ has exactly two nodal domains. We will use this fact to show that $\partial_0 \not \in \Sigma$. Assume $\partial_0 \in \mathrm{int}\,\Sigma$. For any $v \in \mathcal{C}$ we have that $\varphi_v(\partial_0) = 0$. Since $\Sigma$ has genus zero and $\varphi_v$ has exactly two nodal domains, the Proposition \ref{multiplicidadesigma1prop4} guarantees that the order of vanishing of $\varphi_v$ in $\partial_0$ is at most one. So, there is a neighborhood of $\partial_0$ where we can write $\varphi_v\circ\exp_{\partial_0} = P$ with $P$ a homogeneous polynomial of degree one and $\exp_{\partial_0}$ the exponential map of $\Sigma$ in $\partial_0$. For $X \in T_{\partial_0}\Sigma$ we have
$$d(\varphi_v)_{\partial_0}(X) = d(\varphi_v\circ \exp_{\partial_0})_0(X) = P(X),$$
which gives us $\nabla \varphi_v(\partial_0) \not = 0$, since the degree of $P$ is one. Define a linear map $T : \mathcal{C} \rightarrow T_{\partial_0}\Sigma$ by $T(v) = v^T$, where $v^T$ is the projection of $v$ in the tangente space $T_{\partial_0}\Sigma$. Since the dimension of $\mathcal{C}$ is greater than the dimension of $T_{\partial_0}\Sigma$, there is a non-zero vector $v \in \mathcal{C}$ such that $T(v) = 0$, which implies that $\nabla\varphi_v(\partial_0) = 0$, a contradiction. On the other hand, $\partial\Sigma = \Sigma\cap\partial\mathbb{B}^3_\varepsilon(r)$, which implies that $\partial_0$ cannot be in $\partial\Sigma$. So, we conclude that $\partial_0\not\in\Sigma$. In particular, the map $A_\Sigma : \Sigma\rightarrow \Sigma$ does not have fixed points, since $\partial_0$ is the only fixed point of the antipodal map in $\mathbb{B}^3_\varepsilon(r)$.

%Note that this fact alone does not lead to a contradiction with the fact that $\varphi_v$ has exactly two nodal domains, since we can have $\varphi_v(p) \not=0$. But it is true that $\phi_v(\partial_0) = 0$, for all $v \in \mathcal{C}$, so we cannot have $\partial_0\in\mathrm{int}\Sigma$.

%Choose any point $p \in \mathrm{int}\Sigma$ and consider $u$ any eigenfunction with exactly two nodal domains and that satisfies the condition $u(p) = 0$. 

%Step 2: using again the invariance of $\Sigma$, we obtain that $A_\Sigma : \Sigma\rightarrow \Sigma$ is an involutive isometry that reverses the orientation of $\Sigma$. To see this, working with the vector field 
%\begin{center}
%$V_p = \frac{1}{||y||}(-\varepsilon||y||^2\partial_0 + xy)$,
%\end{center} 
%where $p = (x,y)$, we can argue as in \cite[Proposition 8.1]{Fr-Sc} to show that each geodesic starting from $\partial_0$ meets $\Sigma$ at most once. Then, projecting $\Sigma$ into $\partial\mathbb{B}^3_\varepsilon(r)$ we obtain that $A$ reverses the orientation of $\Sigma$.

Step 2: using again the invariance of $\Sigma$, we will show that $A_\Sigma : \Sigma \rightarrow \Sigma$ reverses the orientation of $\Sigma$. As usual, denote $p = (x,y)$. Consider $V$ a unit normal vector field to $\Sigma$ that is tangent to $\mathbb{E}^n_\varepsilon$. Define $f : \Sigma\rightarrow \mathbb{R}$ by $f(p) = \overline{g}_\varepsilon (V_p,\hat{N}_p)$, where $\hat{N}$ is the vector field defined by
$$\hat{N}_p = \frac{1}{||y||}(-\varepsilon||y||^2\partial_0 + xy^i\partial_i).$$

Since $\hat{N}\vert_{\partial\mathbb{B}^3_\varepsilon}$ is the unity outward normal vector field to $\partial\mathbb{B}^3_\varepsilon(r)$ it is immediate $f\vert_{\partial\mathbb{B}^3_\varepsilon} \equiv 0$. We will prove that $f(p) \not = 0$ for any $p \in \mathrm{int}\,\Sigma$. By contradiction, assume $f(p) = 0$ for a $p \in \mathrm{int}\,\Sigma$. Define $S : {\mathcal{C}}\rightarrow T_p\Sigma$ by $S(v) = v^T$, where $v^T$ is the projection of $v$ in the tangent space $T_p\Sigma$. As before, there is a non-zero vector $v \in {\mathcal{C}}$ such that $S(v) = 0$. So, $T_p^\perp\Sigma \subset\mathrm{span}\{v,p\}$ and, in particular, $V_p = v$. Identifying $v$ as $(0,v)$, the condiction $f(p) = 0$ implies that $x = 0$, which cannot happen, or $v\cdot y = 0$, where $p = (x,y)$ and $\cdot$ is the usual inner product of $\mathbb{R}^3$. So, we get that $\varphi_v(p) = 0$. Arguing as in Step 1, there is a homogeneous polynomial of degree one such that $d(\varphi_v)_p(X) = P(X)$, for any $X \in T_p\Sigma$. But $\nabla\varphi_v(p) = 0$, a contradiction. This means that we can choose the vector field $V$ in such a way that $f$ is nonzero on $\mathrm{int}\,\Sigma$. Now, we can use stereographic projection in the case $\varepsilon = 1$, or hyperbolic stereographic projection in the case $\varepsilon = -1$, to bring $\Sigma$ onto $\Sigma'$, a embedded surface in a ball, which satisfies the free boundary condition. Since $f\vert_{\mathrm{int}\,\Sigma}$ is nonzero, we can choose a unit normal vector field $V'$ for $\Sigma'$ such that $p\cdot V'$ is positive on $\mathrm{int}\,\Sigma'$. So, we are in the same situation as in the proof of \cite[Proposition 8.1]{Fr-Sc}. In particular, each geodesic starting from $\partial_0$ meets $\Sigma$ at most once, then, projecting $\Sigma$ into $\partial\mathbb{B}^3_\varepsilon(r)$ we obtain that $A$ reverses the orientation of $\Sigma$.

Step 3: the Lemma \ref{CisAodd} guarantees that for all $v \in \mathcal{C}$ the eingenfunction $\varphi_v$ is $A$-odd. Define $\tilde{\mathcal{C}}$ as the dual space of $\mathcal{C}$ via the canonial isomorphism $v \mapsto \phi_v$. We have that $\tilde{\mathcal{C}}$ is a three-dimensional space formed by $A$-odd functions with exactly two nodal domains and $A_\Sigma : \Sigma\rightarrow \Sigma$ is an involutive isometry without fixed points that reverses the orientation of $\Sigma$. Since $\Sigma$ is an annulus, we conclude, from the Theorem \ref{generale1c}, that $\mathcal{E}_1 = \widetilde{\mathcal{C}}$. This guarantees that the coordinate functions $\varphi_1$, $\varphi_2$ and $\varphi_3$ are eigenfunctions associated with the first eigenvalue of the Steklov problem with frequency. We are then left with the hypotheses of Theorem \ref{LMM}, which concludes the proof. 
\end{proof}

\end{document}